\documentclass[pdftex]{amsart}
\usepackage[utf8x]{inputenc}
\usepackage[T1]{fontenc}
\usepackage{ae,aecompl}

\usepackage{amssymb}
\usepackage{amsmath}
\usepackage{units}

\usepackage[english]{babel}

\usepackage{lineno} % Dans le paquet ubuntu texlive-humanities
\linenumbers

\usepackage{tikz}
\usetikzlibrary{matrix,shapes,arrows,chains}
\usetikzlibrary{calc}
\usetikzlibrary{patterns}
\definecolor{darkgreen}{rgb}{0,0.7,0}
\definecolor{darkred}{rgb}{0.7,0,0}
\definecolor{darkblue}{rgb}{0,0,0.7}
\usepackage{pdflscape}

\usepackage{enumerate}
\usepackage{verbatim}
\usepackage{hyperref}

\usepackage{graphicx}
\usepackage{rotating}
\usepackage{algorithmic}
\usepackage{algorithm}
\usepackage{xspace}

\usepackage{ifpdf}
\ifpdf
\fi

\newtheorem{theorem}{Theorem}[section]
\newtheorem{lemma}[theorem]{Lemma}
\newtheorem{proposition}[theorem]{Proposition}

\newtheorem{problem}[theorem]{Problem}

\newtheorem{remark}[theorem]{Remark}

\newtheorem{example}[theorem]{Example}

\newcommand{\ZZ}{\mathbb{Z}}
\newcommand{\QQ}{\mathbb{Q}}

\newcommand{\KK}{\mathbb{K}}
\newcommand{\kk}{\mathbb{K}}
\newcommand{\GL}{\operatorname{GL}}
\newcommand{\lex}{\operatorname{lex}}

\newcommand{\x}{\mathbf{x}}
\newcommand{\R}{\kk[\mathbf{x}]}
\newcommand{\RG}{\R^G}
\newcommand{\Sym}{\operatorname{Sym}(\mathbf{x})}
\newcommand{\sg}[1][n]{{\mathfrak{S}_{#1}}}

\newcommand{\E}{\mathcal{E}}    % \kk^\sg[n]
\newcommand{\EG}{{\mathcal{E}^G}} % its invariant subalgebra
\newcommand{\OO}{\mathcal{O}}

\newcommand{\q}{\epsilon}
\newcommand{\orbitsum}[1]{o(#1)}

\newcommand{\I}{\mathcal{I}}  %ideal of elementary with e_n -\q in K[x]
\newcommand{\IG}{\mathcal{I}^G} %ideal of elementary with e_n -\q in K[x]^G

\newcommand{\Hilb}[1][n]{H({#1}, z)} % Hilbert series of
\newcommand{\HilbSec}[1][n]{S({#1}, z)} % Secondary invariant series

\newcommand{\SecInv}[1][]{S_{#1}} % set of secondary invariants
\newcommand{\ind}{\hspace{4ex}}
\newcommand{\comm}[1][]{\qquad \text{\footnotesize{\##1}}}

\newcommand{\Sage}{\texttt{Sage}\xspace}
\newcommand{\sagecombinat}{\texttt{Sage-Combinat}\xspace}
\newcommand{\gap}{\texttt{GAP}\xspace}
\newcommand{\magma}{\texttt{Magma}\xspace}
\newcommand{\mupadcombinat}{\texttt{MuPAD-Combinat}\xspace}
\newcommand{\python}{\texttt{Python}\xspace}
\newcommand{\cython}{\texttt{Cython}\xspace}
\newcommand{\singular}{\texttt{Singular}\xspace}

\makeatletter
\newskip\@bigflushglue \@bigflushglue = -100pt plus 1fil

\def\bigcentering{\let\\\@centercr\rightskip\@bigflushglue%
\leftskip\@bigflushglue
\parindent\z@\parfillskip\z@skip}

\makeatother
\usepackage{ifthen}
\newboolean{draft}
\setboolean{draft}{false}
\ifdraft
\newcommand{\TODO}[2][To do: ]{\textcolor{red}{\textbf{#1#2}}}
\newcommand{\INFO}[2][Info: ]{\textcolor{red}{\textbf{#1#2}}}
\else
\newcommand{\TODO}[2][]{}
\newcommand{\INFO}[2][]{}
\fi

\title[An evaluation approach to computing 
  invariants rings]{An evaluation approach to computing\\
  invariants rings of permutation groups}

\author{Nicolas Borie and Nicolas M.~Thi\'ery}
\address{Univ. Paris-Sud, Laboratoire de Math\'ematiques d'Orsay,
         Orsay Cedex, F-91405; CNRS, France}

\begin{document}
\maketitle
%\tableofcontents

\begin{abstract}
  Using evaluation at appropriately chosen points, we propose a
  Gr\"obner basis free approach for calculating the secondary
  invariants of a finite permutation group. This approach allows for
  exploiting the symmetries to confine the calculations into a smaller
  quotient space, which gives a tighter control on the algorithmic
  complexity, especially for large groups. This is confirmed by
  extensive benchmarks using a \Sage implementation.
\end{abstract}

%\TODO{Montrez que cette approche se generalise aux groupes de reflexion complexes.}

\section{Introduction}

% Context and overall goal

Invariant theory has been a rich and central area of algebra ever
since the eighteenth theory, with practical
applications~\cite[\S~5]{Kemper_Derksen.CIT.2002} in the resolution of
polynomial systems with symmetries (see
e.g. \cite{Colin.1997.SolvingSymmetries},
\cite{Gatermann.1990.Symmetry}, \cite[\S~2.6]{Sturmfels.AIT},
\cite{Faugere_Rahmany.2009.SAGBIGroebner}), in effective Galois theory
(see e.g.~\cite{Colin.TIE}, \cite{Abdeljaouad.TIATG},
\cite{Geissler_Kluners.2000.GaloisGroupComputations}), or in discrete
mathematics (see e.g.~\cite{Thiery.AIG.2000,Pouzet_Thiery.IAGR.2001}
for the original motivation of the second author). The literature
contains deep and explicit results for special classes of groups, like
complex reflection groups or the classical reductive groups, as well
as general results applicable to any group. Given the level of
generality, one cannot hope for such results to be simultaneously
explicit and tight in general. Thus the subject was effective early
on: given a group, one wants to \emph{calculate} the properties of its
invariant ring. Under the impulsion of modern computer algebra,
computational methods, and their implementations, have largely
expanded in the last twenty
years~\cite{Kemper.Invar,Sturmfels.AIT,Thiery.CMGS.2001,Kemper_Derksen.CIT.2002,King.2007.secondary,King.2007.minimal}.
However much progress is still needed to go beyond toy examples and
enlarge the spectrum of applications.

An important obstruction is that the algorithms depend largely on
efficient computations in certain quotients of the invariant ring;
this is usually carried out using elimination techniques (Gröbner or
SAGBI-Gröbner bases), but those do not behave well with respect to
symmetries. An emerging trend is the alternative use of evaluation
techniques, for example to rewrite invariants in terms of an existing
generating set of the invariant
ring~\cite{Gaudry_Schost_Thiery.2006,Dahan_Schost_Wu.2009}.
\TODO{Contacter Romain Lebreton pour savoir où il en est!}
% or to compute primary invariants~\cite{CommentIlSAppelleLEtudiantDeGuisti?}.

\textbf{In this paper, and as a test bed, we focus on the problem of
  computing secondary invariants of finite permutation groups in the
  non modular case, using evaluation techniques.}

In Section~\ref{section.preliminaries}, we review some relevant
aspects of computational invariant theory, and in particular discuss
the current limitations due to quotient computations. In
Section~\ref{section.quotient}, we give a new theoretical
characterization of secondary invariants in term of their evaluations
on as many appropriately chosen points; this is achieved by
perturbating slightly the quotient, and using the grading to transfer
back results.
In Section~\ref{section.algorithm}, we derive an algorithm for
computing secondary invariants of permutation groups. We establish in
Section~\ref{section.complexity} a worst case complexity bound for
this algorithm. This bound suggests that, for a large enough group
$G$, at least a factor of $|G|$ is gained. This comparison remains
however sloppy since, to the best of our knowledge and due to the
usual lack of fine control on the complexity of Gröbner bases methods,
no meaningful bound exists in the literature for the elimination based
algorithms. Therefore, in Section~\ref{section.benchmarks} we
complement this theoretical analysis with extensive benchmarks
comparing in particular our implementation in \Sage and the
elimination-based implementation in
\singular's~\cite{Singular,King.2007.secondary}. Those benchmarks
suggest a practical complexity which, for large enough groups, is
cubic in the size $n!/|G|$ of the output. And indeed, if the
evaluation-based implementation can be order of magnitudes slower for
some small groups, it treats predictably large groups which are
completely out of reach for the elimination-based implementation. This
includes\TODO{all transitive permutation groups for which $n!/|G|\leq
  1000$ and} an example with $n=14$, $|G|=50,803,200$, and $1716$
secondary invariants.

We conclude, in Section~\ref{section.future}, with a discussion of
%possible generalizations, in particular for subgroups of general
%complex reflection groups, and 
avenues for further improvements.

\section{Preliminaries}
\label{section.preliminaries}

We refer to~\cite{Stanley.1979,Sturmfels.AIT,Cox_al.IVA,Smith.1997,Kemper.1998,Kemper_Derksen.CIT.2002}
for classical literature on invariant theory of finite groups. Parts
of what follows are strongly inspired by~\cite{Kemper.1998}.
\TODO{Add ref in text... if possible}
Let $V$ be a $\kk$-vector space of finite dimension $n$, and $G$ be a
finite subgroup of $\GL(V)$. Tacitly, we interpret $G$ as a group of
$n\times n$ matrices or as a representation on $V$. Two vectors $v$
and $w$ are \emph{isomorphic}, or in the same \emph{$G$-orbit} (for
short \emph{orbit}), if $\sigma\cdot v=w$ for some $\sigma\in G$.

Let $\x:=(x_1,\dots,x_n)$ be a basis of the dual of $V$, and let $\R$
be the ring of polynomials over $V$. The action of $G$ on $V$ extends
naturally to an action of $G$ on $\R$ by $\sigma\cdot p:=p\circ
\sigma^{-1}$. An \emph{invariant polynomial}, or \emph{invariant}, is
a polynomial $p\in K[x_1,\dots,x_n]$ such that $\sigma\cdot p=p$ for
all $\sigma\in G$. The \emph{invariant ring} $\RG$ is the set of
all invariants.  Since the action of $G$ preserves the
degree of polynomials, it is a graded connected commutative algebra:
% , $I(G)$ decomposes into the direct sum of its homogeneous components:
$\RG=\bigoplus_{d\geq 0} \RG_d$, with $\RG_0\approx \kk$. We write
$\RG_+=\bigoplus_{d> 0} \RG_d$ for the positive part of the
invariant ring. The \emph{Hilbert series} of $\RG$ is the
generating series of its dimensions:
\begin{displaymath}
 \Hilb[\RG] := \sum_{d=0}^\infty z^d \dim \RG_d\,.
\end{displaymath}
It can be calculated using Molien's formula:
\begin{displaymath}
  \Hilb[\RG] = \frac{1}{|G|} \sum_{M \in G} \frac{1}{\det(\operatorname{Id} - zM)}\,.
\end{displaymath}
This formula reduces to Pólya enumeration for permutation
groups. Furthermore, the summation can be taken instead
over conjugacy classes of $G$, which is relatively cheap in practice.

A crucial device is the Reynolds operator:
\begin{displaymath}
  \begin{array}{clll}
    R: & \R &\longrightarrow &\RG\\
       & p  &\longmapsto & \frac{1}{|G|} \sum_{g\in G} g.p\,,
  \end{array}
\end{displaymath}
which is both a graded projection onto $\RG$ and a morphism of
$\RG$-module. Note that its definition requires $\operatorname{char}
\kk$ not to divide $|G|$, which we assume from now on (non-modular
case).

Hilbert's fundamental theorem of invariant theory states that $\RG$
is finitely generated: there exists a finite set $S$ of invariants
such that any invariant can be expressed as a polynomial combination
of invariants in $S$. We call $S$ a \emph{generating set}. If no
proper subset of $S$ is generating, $S$ is a \emph{minimal generating
  set}. Since $\RG$ is finitely generated, there exists a degree
bound $d$ such that $\RG$ is generated by the set of all
invariants of degree at most $d$. We denote by $\beta(\RG)$ the
\emph{smallest degree bound}. Noether proved that $\beta(\RG)\leq
|G|$.

Thanks to the grading, for $M$ a set of homogeneous invariants, the
following properties are equivalent:
\begin{enumerate}[(i)]
\item $M$ is a minimal generating set for $\RG$;
\item $M$ is a basis of the quotient $\RG / {\RG_+}^2$.
\end{enumerate}
Therefore, even though the generators in $M$ are non canonical, the
number of generators of a given degree $d$ in $M$ is: it is given by
the dimension of the component of that degree in the graded quotient
$\RG / {\RG_+}^2$. There is no known algorithm to compute those
dimensions, or even just $\beta(\RG)$, without computing explicitly a
minimal generating set.

The previous properties give immediately a naive algorithm for
computing an homogeneous minimal generating set, calculating degree by
degree in the finite dimensional quotient up to Noether's bound. There
are however two practical issues. The first one is that Noether's
bound is tight only for cyclic groups; in general it is very dull,
possibly by orders of magnitude. The second issue is how to compute
efficiently in the given quotient. We will get back to it.

By a celebrated result of Shepard, Todd, Chevalley, and Serre, $\RG$
is a polynomial algebra if and only if $G$ is a complex reflection
group. In all other cases, there are non trivial relations (also
called syzygies) between the generators; however $\RG$ remains
\emph{Cohen-Macaulay}. Namely, a set of $m$ homogeneous invariants
$(\theta_1,\dots,\theta_n)$ of $\RG$ is called a \emph{homogeneous
  system of parameters} or, for short, a \emph{system of parameters}
if the invariant ring $\RG$ is finitely generated over its subring
$\kk[\theta_1,\dots,\theta_n]$. That is, if there exist a finite
number of invariants $(\eta_1,\dots,\eta_t)$ such that the invariant
ring is the sum of the subspaces
$\eta_i.\kk[\theta_1,\dots,\theta_n]$.  By Noether's normalization
lemma, there always exists a system of parameters for $\RG$. Moreover,
$\RG$ is \emph{Cohen-Macaulay}, which means that $\RG$ is a
free-module over any system of parameters. Hence, if the set
$(\eta_1,\dots,\eta_t)$ is minimal for inclusion, $\RG$ decomposes
into a direct sum:
\begin{displaymath}
  \RG=\bigoplus_{i=1}^t \eta_i . \kk[\theta_1,\dots,\theta_n].
\end{displaymath}
This decomposition is called a \emph{Hironaka decomposition} of the
invariant ring. The $\theta_i$ are called \emph{primary invariants},
and the $\eta_i$ \emph{secondary invariants} (in algebraic
combinatorics literature, the $\theta_i$ are some times called
\emph{quasi-generators} and the $\eta_i$
\emph{separators}~\cite{Garsia_Stanton.1984}). It should be emphasized
that primary and secondary invariants are not uniquely determined, and
that being a primary or secondary invariant is not an intrinsic
property of an invariant $p$, but rather express the role of $p$ in a
particular generating set.

The primary and secondary invariants together form a generating set,
usually non minimal. From the degrees $(d_1,\dots,d_n)$ of the
primary invariants $(\theta_1,\dots,\theta_n)$ and the Hilbert series
we can compute the number $t$ and the degrees $(d'_1,\dots,d'_t)$ of the
secondary invariants $(\eta_1,\dots,\eta_t)$ by the formula:
\begin{displaymath}
  \label{eq:degres_secondaires}
  z^{d'_1}+\dots+z^{d'_t}=(1-z^{d_1})\cdots(1-z^{d_n}) H(\RG,z)\,.
\end{displaymath}
We denote this polynomial by $\HilbSec[\RG]$.
% and call it \emph{secondary invariants series}. \TODO{or generating
% polynomial for the secondary invariants???; see the literature}
%
Assuming $d_1\le\dots\le d_n$ and $d'_1\le\dots\le d'_t$, it can be
proved that:
\begin{displaymath}
  \label{eq.max_secondaires}
  \begin{gathered}
    t   = \frac{d_1\cdots d_n}{|G|}\,,\qquad
    d'_t = d_1+\dots+d_n - n - \mu\,,\qquad
    \beta(\RG)\le \max(d_n,d'_t)\,,
  \end{gathered}
\end{displaymath}
where $\mu$ is the smallest degree of a polynomial $p$ such that
$\sigma\cdot p=\det(\sigma) p$ for all $\sigma\in G$ \cite[Proposition
3.8]{Stanley.1979}.

\TODO{bad hyphenation for symmetric}

For example, if $G$ is the symmetric group $\sg$, the $n$ elementary
symmetric polynomials (or the $n$ first symmetric power sums) form a
system of parameters, $t=1$, $d'_t=0$ and $\eta_1=1$. This is
consistent with the fundamental theorem of symmetric polynomials.
More generally, if $G$ is a permutation group, the elementary
symmetric polynomials still form a system of parameters: $\RG$ is a
free module over the algebra $\Sym=\R^{\sg}$ of symmetric polynomials. It
follows that:
\begin{gather*}
  t   = \frac{n!}{|G|}\,,\qquad
  d'_t = \binom n 2 - \mu\,,\qquad
  \beta(\RG)\le \binom{n}{2}\,.
\end{gather*}

For a review of algorithms to compute primary invariants with minimal
degrees, see~\cite{Kemper_Derksen.CIT.2002}. They use Gröbner bases,
exploiting the property that a set $\Theta_1,\dots,\Theta_n$ of $n$
homogeneous invariants forms a system of parameters if and only if
$\x=0$ is the single solution of the system of equations
$\Theta_1(\x)=\cdots=\Theta_n(\x)=0$ (see
e.g.~\cite[Proposition~3.3.1]{Kemper_Derksen.CIT.2002}).

\TODO{Recently the use of evaluation techniques has been
  explored%~\cite{Lebreton_Schost}.
}

We focus here on the second step: we assume that primary invariants
$\Theta_1,\dots,\Theta_n$ are given as input, and want to compute
secondary invariants. This is usually achieved by using the following
proposition to reduce the problem to linear algebra.
\begin{proposition}
  \label{proposition.secondary}
  Let $\Theta_1,\dots,\Theta_n$ be primary invariants and
  $S:=(\eta_1,\dots,\eta_t)$ be a family of homogeneous invariants with
  the appropriate degrees. Then, the following are equivalent:
  \begin{enumerate}[(i)]
  \item $S$ is a family of secondary invariants;
  \item $S$ is a basis of the quotient $\RG / \langle
    \Theta_1,\dots,\Theta_n\rangle_{\RG}$;
  \item $S$ is free in the quotient $\R / \langle
    \Theta_1,\dots,\Theta_n\rangle_{\R}$.
  \end{enumerate}
\end{proposition}

The central problem is how to compute efficiently inside one of the quotients
$\R / \langle \Theta_1,\dots,\Theta_n\rangle_{\R}$ or $\RG / \langle
\Theta_1,\dots,\Theta_n\rangle_{\RG}$. Most algorithms rely on $(iii)$
using normal form reductions w.r.t. the Gröbner basis for
$\Theta_1,\dots,\Theta_n$ which was calculated in the first step to
prove that they form a system of parameters. The drawback is that
Gröbner basis and normal form calculations do not preserve symmetries;
hence they cannot be used to confine the calculations into a small
subspace of $\R / \langle \Theta_1,\dots,\Theta_n\rangle_{\R}$.
Besides, even the Gröbner basis calculation itself can be intractable
for moderate size input ($n=8$)\TODO{the benchmark should mark all
  groups for which this blows up!} in part due to the large
multiplicity ($d_1\cdots d_n$) of the unique root $\x=0$ of this
system.
\TODO{(Note of the referee : think to the product of initial degree of a Gröbner basis of a galoisian ideal of the polynomial
$x^n$ when the group is a permutation group)}

An other approach is to use $(ii)$. Then, in many cases, one can make
use of the symmetries to get a compact representation of invariant
polynomials. For example, if $G$ is a permutation group, an invariant
can be represented as a linear combination of orbitsums instead of a
linear combination of monomials, saving a factor of up to $|G|$ (see
e.g.~\cite{Thiery.CMGS.2001}). Furthermore, one can use SAGBI-Gröbner
bases (an analogue of Gröbner basis for ideals in subalgebras of
polynomial rings) to compute in the quotient
(see~\cite{Thiery.CMGS.2001,Faugere_Rahmany.2009.SAGBIGroebner}). However
SAGBI and SAGBI-Gröbner basis tend to be large (in fact, they are
seldom finite, see~\cite{Thiery_Thomasse.SAGBI.2002}), even when
truncated.

\TODO{We would need some explicit figures for the argument to be
  convincing; e.g. give the size per degree of the SAGBI-Gröbner basis
  of the ideal generated by $\Sym^+$ for some permutation group $G$.}

In both cases, it is hard to derive a meaningful bound on the
complexity of the algorithm, by lack of control on the behavior of the
(SAGBI)-Gröbner basis calculation.\TODO{Here we would need to discuss
with e.g. Schost to argue properly}
%
% The symmetric group $\sg[n]$ has a
% natural action on the ring $\R$ of multivariate polynomial in
% $\mathbf{x} = (x_1, x_2, \dots , x_n)$. Each simple transposition
% $s_i \in \sg[n]$ swaps the variables $x_i$ and $x_{i+1}$. Extending the
% action of simple transposition, a permutation $\sigma \in \sg[n]$ acts on a
% polynomial $P \in \R$ as follow
% \begin{displaymath}
% \sigma(P(x_1, x_2, \dots , x_n)) = P(x_{\sigma(1)}, x_{\sigma(2)}, \dots , x_{\sigma(n)})
% \end{displaymath}
% We will say that a polynomial $P$ is invariant under the action of a permutation
% group $G$ if for all permutation $\sigma \in G : \sigma(P) = P$. A linear
% combination of invariants being also invariant and a product of two invariant
% polynomials still being invariant, we call $\RG$ the invariant ring of the
% permutation group $G$, which is a subalgebra of $\R$. In the case $G$ is
% the symmetric group $\sg[n]$ of degree $n$, the ring of invariant is well known
% as $\Sym$, the algebra of symmetric polynomials.
%
In the following section, we propose to calculate in the quotient $\RG
/ \langle \Theta_1,\dots,\Theta_n\rangle_{\RG}$ using instead
evaluation techniques.

\TODO{Better title!}

\section{Quotienting by evaluation}
\label{section.quotient}

\newcommand{\rhoroot}{\boldsymbol\rho}

Recall that, in the good cases, an efficient mean to compute modulo an
ideal is to use evaluation on its roots.
\begin{proposition}
  \label{proposition.semisimple_quotient}
  Let $P$ be a system of polynomials in $\R$ admitting a finite set
  $\rhoroot_1,\dots,\rhoroot_r$ of multiplicity-free roots, and let $I$ be the
  dimension $0$ ideal they generate. Endow further $\kk^r$ with the
  pointwise (Hadamard) product. Then, the evaluation map:
  \begin{displaymath}
    \begin{array}{clll}
      \Phi: & \R &\longrightarrow &\kk^r\\
            & p  &\longmapsto  &(p(\rhoroot_1),\dots,p(\rhoroot_r))
    \end{array}
  \end{displaymath}
  induces an isomorphism of algebra from $\R/I$. In particular, $\R/I$
  is a semi-simple basic algebra, a basis of which is given by the $r$
  idempotents $(p_i)_{i=1,\dots,r}$ which satisfy
  $p_i(\rhoroot_j)=\delta_{i,j}$; those idempotents can be constructed
  by multivariate Lagrange interpolation\TODO{find a ref}, or using
  the Buchberger-Möller algorithm~\cite{Buchberger_Moller.1982}.
\end{proposition}

This proposition does not apply directly to the ideal $\langle
\Theta_1,\dots,\Theta_n\rangle$ because it has a single root with a very
high multiplicity $d_1\dots d_n$. The central idea of this paper is to blowup this
single root by considering instead the ideal $\langle
\Theta_1,\dots,\Theta_{n-1},\Theta_n-\q\rangle$, where $\q$ is a non
zero constant, and then to show that the grading can be used to
transfer back the result to the original ideal, modulo minor
complications. This approach is a priori general: assuming the field
is large enough, the ideal
$\Theta_1,\dots,\Theta_n$ can always be slightly perturbed to admit
$d_1\cdots d_n$ multiplicity-free roots; those roots are obviously
stable under the action of $G$, and can be grouped into orbits. 
Yet it can be non trivial to compute and describe those roots.

For the sake of simplicity of exposition, we assume from now on that
$G$ is a permutation group, that $\Theta_1,\dots,\Theta_n$ are the
elementary symmetric functions $e_1,\dots,e_n$, and that
$\q=(-1)^{n+1}$. Finally we assume that the ground field $\kk$
contains the $n$-th roots of unity; this last assumption is reasonable
as, roughly speaking, the invariant theory of a group depends only on
the characteristic of $\kk$. With those assumptions, the roots
$\rhoroot_i$ take a particularly nice and elementary form, and open
connections with well know combinatorics. Yet we believe that this
case covers a wide enough range of groups (and applications) to
contain all germs of generality. In particular, the results presented
here should apply mutatis mutandis to any subgroup $G$ of a complex
reflection group.

\begin{remark}
  Let $\rho$ be a $n$-th primitive root of unity, and set
  $\rhoroot:=(1,\rho,\dots,\rho^{n-1})$. Then,
  $e_1(\rhoroot)=\cdots=e_{n-1}(\rhoroot) = 0$ and $e_n(\rhoroot)=\q$.
  % \begin{displaymath}
  %   e_1(\rhoroot)=\cdots=e_{n-1}(\rhoroot) = 0 \qquad \text{ and } \qquad
  %   e_n(\rhoroot)=\q\,.
  % \end{displaymath}
\end{remark}
\begin{proof}
  Up to sign, $e_i(\rhoroot)$ is the $i$-th coefficient of the
  polynomial
  \begin{displaymath}
    (X^n - 1) = \displaystyle\prod_{i = 0}^{n-1} (X - \rho^i)\,.\qedhere
  \end{displaymath}
\end{proof}

For $\sigma\in \sg$, write $\rhoroot_\sigma:=\sigma\cdot\rhoroot$ the
permuted vector. It follows from the previous remark that the orbit
$(\rhoroot_\sigma)_{\sigma\in\sg}$ of $\rhoroot$ gives all the roots
of the system
\begin{displaymath}
  e_1(\x)=\cdots=e_{n-1}(\x)=e_n(\x)-\q=0\,.
\end{displaymath}
Let $\I$ be the ideal generated by $e_1,\dots,e_{n-1},e_n-\q$ in $\R$,
that is the ideal of symmetric relations among the roots of the
polynomial $X^n-1$; it is well known that the quotient $\R/\I$ is of
dimension $n!$\TODO{ref?}. We define the evaluation map $\Phi: p \in
\R \mapsto (p(\rhoroot_\sigma))_\sigma$ as in
Proposition~\ref{proposition.semisimple_quotient} to realize the
isomorphism from $\R/\I$ to $\E = \kk^{\sg}$.

Obviously, the evaluation of an invariant polynomial $p$ is constant
along $G$-orbits. This simple remark is the key for confining the
quotient computation into a small subspace of dimension $n!/|G|$,
which is precisely the number of secondary invariants. Let $\EG$ be
the subalgebra of the functions in $\E$ which are constant along
$G$-orbits. Obviously, $\EG$ is isomorphic to $\kk^L$ where $L$ is any
transversal of the right cosets in $\sg/G$. Let $\IG$ be the ideal
generated by $(e_1,\dots,e_{n-1},e_n-\q)$ in $\RG$; as the notation
suggests, it is the subspace of invariant polynomials in $\I$.
\begin{remark}\label{remark.isoquotientinvariant}
  The restriction of $\Phi$ on $\RG$, given by:
  \begin{displaymath}
    \begin{array}{clll}
      \Phi: & \RG &\longrightarrow &\EG\\
            & p  &\longmapsto  &(p(\rhoroot_\sigma))_{\sigma\in L}
    \end{array}
  \end{displaymath}
  is surjective and induces an algebra isomorphism between $\RG/\IG$
  and $\EG$.
\end{remark}
\begin{proof}
  For each evaluation point $\rhoroot_\sigma$, $\sigma\in L$, set
  \begin{displaymath}
    \overline p_{\rhoroot_\sigma} := \sum_{\tau\in \sigma G} p_{\rhoroot_\tau}\,,
  \end{displaymath}
  where $p_{\rhoroot_\tau}$ is the Lagrange interpolator of
  Proposition~\ref{proposition.semisimple_quotient}. Then, their
  images $(\Phi(\overline p_{\rhoroot_\sigma}))_{\sigma\in L}$ are
  orthogonal idempotents and, by dimension count, form a basis of $\EG$.
\end{proof}
From now on, we call evaluation points the family
$(\rhoroot_\sigma)_{\sigma\in L}$.

We proceed by showing that the grading can be used to compute modulo
the original ideal $\langle e_1,\dots,e_n\rangle$, modulo minor
complications.
\begin{lemma}\label{keylemma}
  Let $G$ be a subgroup of $\sg$ and $\kk$ be a field of
  characteristic $0$ containing a primitive $n$-th root of unity. Let
  $\SecInv$ be a set of secondary invariants w.r.t. the primary invariants
  $e_1,\dots,e_n$, and write $\langle \SecInv \rangle_\kk$ for the vector
  space they span (equivalently, one could choose a graded
  supplementary of the graded ideal $\langle e_1,\dots,e_n\rangle$ in
  $\RG$). Write $\SecInv[d]$ for the secondary invariants of degree
  $d$. Then,
  \begin{displaymath}
    \begin{array}{rl}
      \text{for } 0 \leqslant d < n : & \Phi(\R_d^G) = \Phi(\langle \SecInv[d]\rangle_{\kk})\,, \\
      \text{for } d \geqslant n : & \Phi(\R_d^G) = \Phi(\langle \SecInv[d]\rangle_{\kk}) \oplus \Phi(\R_{d-n}^G)\,.\\
    \end{array}
  \end{displaymath}
  In particular, $\Phi$ restricts to an isomorphism from $\langle \SecInv \rangle_\kk$ to $\EG$.
\end{lemma}
\begin{proof}
  %[Sketch of proof]
  %Apply $\Phi$ degree by degree on the Hironaka decomposition.
  For ease of notation, we write the Hironaka decomposition by
  grouping the secondary invariants by degree:
  \begin{displaymath}
    \RG = \bigoplus_{i=1}^t \eta_i \kk[e_1, \dots , e_n] =
    \bigoplus_{d=0}^{d_{\max}} \langle \SecInv[d] \rangle_\kk \kk[e_1, \dots , e_n]\,,
  \end{displaymath}
  where $d_{\max}$ is the highest degree of a secondary invariant. Then, using that
  \begin{displaymath}
    \begin{array}{l}
      \Phi(e_1) = \cdots = \Phi(e_{n-1}) = 0_\EG
      \quad \text{ and } \quad
      \Phi(e_n) = 1_\EG\,, \\
    \end{array}
  \end{displaymath}
  we get that $\Phi(\kk[e_1,\dots,e_n])=\Phi(\kk[e_n])=\kk.1_\EG$, and thus:
  \begin{displaymath}
    \EG = \Phi(\RG) = \sum_{d=1}^{d_{\max}} \Phi(\langle \SecInv[d] \rangle_\kk) \Phi(\kk[e_1,\dots,e_n]) = \sum_{d=1}^{d_{\max}} \Phi(\langle \SecInv[d] \rangle_\kk)\,,
  \end{displaymath}
  where, by dimension count, the sum is direct.
  Using further that $e_n$ is of degree $n$:
  \begin{alignat*}{4}
    \Phi(\R_d^G) &= \Phi(\langle \SecInv[d]\rangle_\kk) &&+ \Phi(\langle \SecInv[d-n]\rangle_\kk e_n) &&+ \Phi(\langle\SecInv[d-2n]\rangle_{\kk} e_n^2) &&+ \cdots\\
      &= \Phi(\langle \SecInv[d]\rangle_\kk) &&\oplus \Phi(\langle \SecInv[d-n]\rangle_\kk) &&\oplus \Phi(\langle\SecInv[d-2n]\rangle_{\kk}) &&\oplus \cdots
  \end{alignat*}
  The desired result follows by induction.
\end{proof}

In practice, this lemma adds to
Proposition~\ref{proposition.secondary} two new equivalent
characterizations of secondary invariants:
\begin{theorem}\label{theorem.secondary.evaluation}
  Let $G\subset \sg[n]$ be a permutation group, take $e_1,\dots,e_n$
  as primary invariants, and let $S=(\eta_1,\dots,\eta_t)$ be a family
  of homogeneous invariants with the appropriate degrees. Then, the
  following are equivalent:
  \begin{enumerate}[(i)]
  \item $S$ is a set of secondary invariants;
  \item[(iv)] $\Phi(S)$ forms a basis of $\EG$;
  \item[(v)] The elements of $\Phi(\SecInv[d])$ are linearly independent in
    $\EG$, modulo the subspace
    \begin{displaymath}
      \sum_{0\leq j < d, \ n \,|\, d-j} \langle \Phi(S_j)\rangle_\kk\,.
    \end{displaymath}
  \end{enumerate}
  Furthermore, when any, and therefore all of the above hold, the sum
  in $(v)$ is a direct sum.
\end{theorem}
\begin{proof}
  Direct application of Lemma~\ref{keylemma}, together with recursion
  for the direct sum.
\end{proof}

\begin{example}
  Let $G=\mathcal A_3=\langle (1,2,3)\rangle$ be the alternating group
  of order $3$. In that case, $\rho$ is the third root of unity $j$,
  and $\KK=\QQ(j)=\QQ\oplus_\QQ \QQ.j \oplus_\QQ \QQ.j^2$.
  We are looking for $n!/|G|=2$ secondary invariants, whose degree
  are given by the numerator of the Hilbert series:
  \begin{displaymath}
    \Hilb[\RG] = \frac 1 3 \left( \frac{1}{(1-z)^3} + 2 \frac{1}{(1-z^3)} \right) =
    \frac{1+z^3}{(1-z)(1-z^2)(1-z^3)}
  \end{displaymath}
  Simultaneously, the $\sg[n]$-orbit of $(1,\rho,\rho^2)$ splits in
  two $G$-orbits. We can, for example, take as evaluation points the
  two $G$-orbit representatives $\rhoroot_{()} = (1,\rho,\rho^2)$ and
  $\rhoroot_{(2,1)}=(\rho,1,\rho^2)$, and the evaluation morphism is
  given by:
  \begin{displaymath}
    \begin{array}{clll}
      \Phi: & \RG &\longrightarrow &\EG=\KK^2\\
            & p  &\longmapsto  &(p(\rhoroot_{()}), p(\rhoroot_{(1,2)})))
    \end{array}
  \end{displaymath}
  For example, $\Phi(1)=\Phi(e_3)=(1,1)$, whereas
  $\Phi(e_1)=\Phi(e_2)=0$. Let us evaluate the orbitsum of the
  monomial $x_1^2x_2 = \x^(2,1,0)$, using
  Remark~\ref{remark.evaluation.orbitsum}:
  \begin{alignat*}{3}
    &\orbitsum{\x^{(2,1,0)}}(\rhoroot_{()}) &=&
    j^{\langle (2,1,0), (0,1,2)\rangle} +
    j^{\langle (2,1,0), (1,2,0)\rangle} +
    j^{\langle (2,1,0), (2,0,1)\rangle} &=&
    3j\,,\\
    &\orbitsum{\x^{(2,1,0)}}(\rhoroot_{(1,2)}) &=&
    j^{\langle (2,1,0), (1,0,2)\rangle} +
    j^{\langle (2,1,0), (0,2,1)\rangle} +
    j^{\langle (2,1,0), (2,1,0)\rangle} &=&
    3j^2\,.
  \end{alignat*}
  That is $\Phi(\orbitsum{x_1^2x_2})=3.(j,j^2)$. It follows that:
  \begin{align*}
    \Phi(\RG_0)&=\KK.(1,1)\\
    \Phi(\RG_1)&= \Phi(\RG_2) = \{(0,0)\}\\
    \Phi(\RG_3)&=\KK.(1,1)\oplus \KK.(3,3) =
    \KK.\Phi(1) \oplus \KK.\Phi(\orbitsum{x_1^2x_2})\,.
  \end{align*}
  In particular, $1$ and $\orbitsum{x_1^2x_2}$ are two secondary
  invariants, both over $\KK$ or $\QQ$:
  \begin{displaymath}
    \RG = \Sym \oplus \Sym.\orbitsum{x_1^2x_2}\,.
  \end{displaymath}

  We consider now the two extreme cases. For $G=\sg[n]$, there is a
  single evaluation point and a single secondary invariant $1$; and
  indeed, $\Phi(1)=(1)$ spans $\Phi(\RG) = \KK$. Take now $G=\{()\}$
  the trivial permutation group on $n$ points. Then, the evaluation
  points are the permutations of $(1,j,j^2,\dots,j^{n-1})$. In that
  case, Theorem~\ref{theorem.secondary.evaluation} states in
  particular that the matrix $(j^{\langle m, \sigma})_{m,\sigma}$,
  where $m$ and $\sigma$ run respectively through the integer vectors
  below the staircase and through $\sg[n]$, is non singular.\TODO{is
    this obvious?}
\end{example}

\section{An algorithm for computing secondary invariants by evaluation}
\label{section.algorithm}

Algorithm~\ref{algorithm} is a straightforward adaptation of the
standard algorithm to compute secondary invariants in order to use the
evaluation morphism $\Phi$ together with
Theorem~\ref{theorem.secondary.evaluation}.
\begin{algorithm}
  \def\I{\operatorname{I}}
  \def\E{\operatorname{E}}
  \def\S{\operatorname{S}}
  \def\leftarrow{=}
  \caption{Computing secondary invariants and irreducible secondary
    invariants of a permutation group $G$, w.r.t. the symmetric
    functions as primary invariants, and using the evaluation morphism
    $\Phi$.}
  \label{algorithm}
  \raggedright

  We assume that the following have been precomputed from the Hilbert
  series:
  \begin{itemize}
  \item $s_d$: the number of secondary invariants of degree $d$\\
    (this is the coefficient of degree $d$ of $\HilbSec[\RG]$)
    
  \item $e_d$: the dimension of $\dim \Phi(\RG_d)$\\
    (this is $s_d$ if $d<n$ and $e_{d-n}+s_d$ otherwise)
  \end{itemize}

  At the end of each iteration of the main loop:
  \begin{itemize}
  \item $\S_d$ is a set $\SecInv[d]$ of secondary invariants of degree $d$;
  \item $\I_d$ is a set of irreducible secondary invariants of degree $d$;
  \item $E_d$ models the vector space $\Phi(\RG_d)$.
  \end{itemize}

  Code, in pseudo-\python syntax:
\begin{displaymath}
  \begin{small}
  \begin{array}{l}
    \textbf{def }\text{SecondaryInvariants(G)} : \\
    \ind \textbf{for } d \in \{0, 1, 2, \dots , \deg(\HilbSec[\RG])\}: \\
    \ind \ind \I_d \leftarrow \{\}\\
    \ind \ind \S_d \leftarrow \{\}\\
    \ind \ind \textbf{if } d \geqslant n :\\
    \ind \ind \ind \E_d \leftarrow \E_{d-n} \comm[Defect of direct sum of Theorem~\ref{theorem.secondary.evaluation}]\\
    \ind \ind \textbf{else} : \\
    \ind \ind \ind \E_d \leftarrow \{\vec 0\} \\
    \ind \ind \text{\# Consider all products of secondary invariants of lower degree} \\
    \ind \ind \textbf{for } (\eta, \eta') \in \S_k \times \I_l \text{ with } k+l = d : \\
    \ind \ind \ind \textbf{if } \Phi(\eta\eta') \notin \E_d:\\
    \ind \ind \ind \ind \S_d \leftarrow \S_d \cup \{\eta\eta'\} \\
    \ind \ind \ind \ind \E_d \leftarrow \E_d \oplus \kk. \Phi(\eta\eta') \\
    \ind \ind \text{\# Complete with orbitsums of monomials under the staircase} \\
    \ind \ind \textbf{for } m \in \text{CanonicalMonomialsUnderStaircaseOfDegree(d)} : \\
    \ind \ind \ind \textbf{if } \dim \E_d == e_d: \\
    \ind \ind \ind \ind \textbf{break} \comm[All secondary invariants were found]\\
    \ind \ind \ind \eta \leftarrow \text{OrbitSum}(m) \\
    \ind \ind \ind \textbf{if } \Phi(\eta) \notin \E_d : \\
    \ind \ind \ind \ind \I_d \leftarrow \I_d \cup \{\eta\} \\
    \ind \ind \ind \ind \S_d \leftarrow \S_d \cup \{\eta\} \\
    \ind \ind \ind \ind \E_d \leftarrow \E_d \oplus \kk.\Phi(\eta) \\
    \ind \textbf{return } (\{\S_0, \S_1, \dots \}, \{\I_0, \I_1, \dots \})
  \end{array}
  \end{small}
\end{displaymath}
\end{algorithm}

For the sake of the upcoming complexity analysis, we now detail how
the required new invariants in each degree can be generated and
evaluated in the case of a permutation group.

It is well known that the ring $\R$ is a free $\Sym$-module of rank
$n!$. It admits several natural bases over $\Sym$, including the
Schubert polynomials, the descent monomials, and the monomials under
the staircase. We focus on the later. Namely, encoding a monomial
$m=\x^\alpha$ in $\R$ by its exponent vector $\alpha=(\alpha_1, \dots
, \alpha_n)$, $m$ is \emph{under the staircase} if $\alpha_i \leqslant
n-i$ for all $1 \leqslant i \leqslant n$. Given a permutation group $G
\subset \sg[n]$, a monomial $m$ is \emph{canonical} if $m$ is maximal
in its $G$-orbit for the lexicographic order: $\sigma(m)
\leqslant_{\lex} m$, $\forall \sigma \in G$. The following lemma is a
classical consequence of the Reynolds operator being a $\RG$-module
morphism.
\begin{lemma}
% <<<<<<< local
%   Orbit sums over $G$ of monomials under staircase, canonical under
%   the action a permutation group $G$ generate $\RG$ as a $\Sym$-module.
%   \begin{proof}
%     Let $P \in \RG$ be an invariant polynomial. As monomials under staircase
%     is a basis of multivariate polynomials over $\Sym$,
%     there exist a unique family $\{f_M\}_{M \leqslant M_{\omega}}$ of
%     symmetric polynomials such that
%     \begin{displaymath}
%       P = \sum_{M \leqslant M_{\omega}} f_M M
%     \end{displaymath}
%     Now calculating the average sum over the group(the Reynold's operator) for each
%     side of the previous equality, we get
%     \begin{displaymath}
%       \begin{array}{rl}
%         P = \frac{1}{|G|} \displaystyle\sum_{\sigma \in G} P = &
%         \frac{1}{|G|} \displaystyle\sum_{\sigma \in G} \displaystyle\sum_{M \leqslant M_{\omega}} f_M M \\
%         P = &
%         \displaystyle\sum_{M \leqslant M_{\omega}} \frac{f_M}{|G|} (\displaystyle\sum_{\sigma \in G} M) \\
%         P = &
%         \displaystyle\sum_{M \leqslant M_{\omega} \atop{M canonical}} (\displaystyle\sum_{N \in orbit(M)} \frac{f_N}{|G|}) (\displaystyle\sum_{\sigma \in G} M) \\
%         P = &
%         \displaystyle\sum_{M \leqslant M_{\omega} \atop{M canonical}} g_M (\displaystyle\sum_{\sigma \in G} M) \\
%       \end{array}
%     \end{displaymath}
%     Thus we obtain $P$ as a linear combination of the good orbit sums whose
%     coefficient $g_M$ are the symmetric polynomial
%     $\displaystyle\sum_{N \in orbit(M)} \frac{f_N}{|G|}$.
%   \end{proof}
% =======
  \label{lemma.spanning}
  Let $M$ be a family of polynomials which spans $\R$ as a
  $\Sym$-module. Then, the set of invariants $\{R(m) \mid m\in
  M\}$ spans $\RG$ as a $\Sym$ module. 

  In particular, taking for $M$ the set of monomials under the
  staircase, one gets that the orbitsums of monomials which are
  simultaneously canonical and under the staircase generate $\RG$ as a
  $\Sym$-module. One can further remove non zero integer partitions
  from this set.
\end{lemma}
\begin{proof}
  Let $p \in \RG$ be an invariant polynomial, and write it as
  $p=\sum_{m\in M} f_m m$, where the $f_m$ are symmetric polynomials.
  Then, using that the Reynolds operator $R$ is a $\RG$-module
  morphism, one gets as desired that:
  \begin{displaymath}
    p = R(p) = R(\sum_{m\in M} f_m m) = \sum_{m\in M} f_m R(m)\,.\qedhere
  \end{displaymath}
\end{proof}

%%%% That's an old proof:
% As a polynomial in $\R$,
%   there exist a unique family $\{f_M\}_{M \leqslant M_{\omega}}$ of
%   symmetric polynomials such that
%     \begin{displaymath}
%       P = \sum_{M \leqslant M_{\omega}} f_M M
%     \end{displaymath}
%     Now applying an average sum over the group(the Reynold's operator) each
%     side of the previous equality, we get
%     \begin{displaymath}
%       \begin{array}{rl}
%         P = \frac{1}{|G|} \displaystyle\sum_{\sigma \in G} P = &
%         \frac{1}{|G|} \displaystyle\sum_{\sigma \in G} \displaystyle\sum_{M \leqslant M_{\omega}} f_M M \\
%         P = &
%         \displaystyle\sum_{M \leqslant M_{\omega}} \frac{f_M}{|G|} (\displaystyle\sum_{\sigma \in G} M) \\
%         P = &
%         \displaystyle\sum_{M \leqslant M_{\omega} \atop{M canonical}} (\displaystyle\sum_{N \in orbit(M)} \frac{f_N}{|G|}) (\displaystyle\sum_{\sigma \in G} M) \\
%         P = &
%         \displaystyle\sum_{M \leqslant M_{\omega} \atop{M canonical}} g_M (\displaystyle\sum_{\sigma \in G} M) \\
%       \end{array}
%     \end{displaymath}
%     Thus we obtain $P$ as a linear combination of the good orbit sums whose
%     coefficient $g_M$ are the symmetric polynomial
%     $\displaystyle\sum_{N \in orbit(M)} \frac{f_N}{|G|}$.

\begin{remark}
  \label{remark.canonical}
  The canonical monomials under the staircase can be iterated through
  efficiently using orderly generation~\cite{Read.1978,McKay.1998} and
  a strong generating system of the group
  $G$~\cite{Seress.2003.PermiutationGroupAlgorithms}; the complexity
  of this iteration can be safely bounded above by $\OO(n!)$, though
  in practice it is much better than that (see
  Figures~\ref{benchmarks.cannonics.average}
  and~\ref{benchmarks.cannonics.bound}, and~\cite{Borie.2011.Thesis}
  for details).
\end{remark}

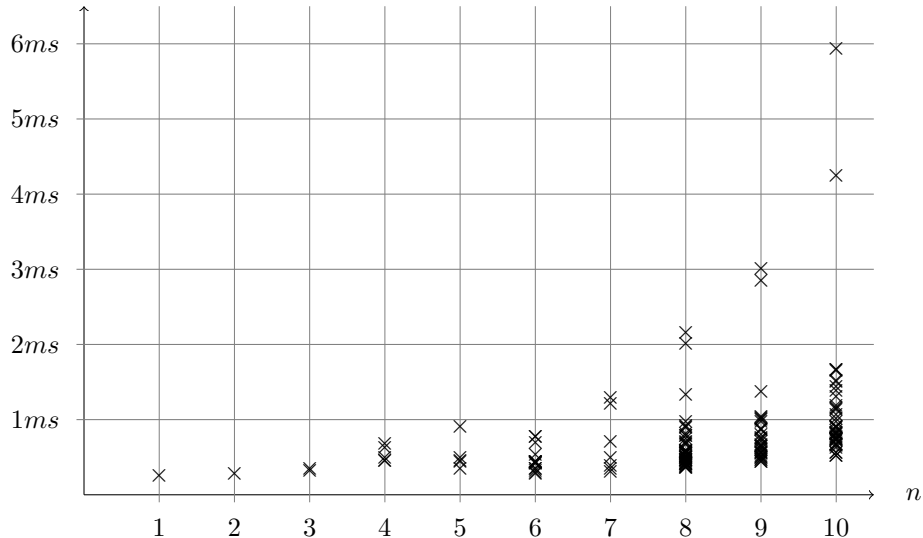
\begin{figure}
  \label{benchmarks.cannonics.average}
  \centering
    \begin{tikzpicture}
    \tikzstyle{bull}=[minimum size=4pt, inner sep=0pt] 
    \tikzstyle{vertical}=[very thin, color=gray] 
    \tikzstyle{degree}=[thin] 
    \tikzstyle{horizontal}=[-, color=gray] 
      % axis 
      \draw[->] (0,0) -- (10.5,0); 
      \draw[->] (0,0) -- (0,6.5); 

      \draw (10.8, 0) node[right] {$n$}; 
%      \draw (0, 6.8) node[above] {$temps / C(G)$}; 

\draw[vertical] (1,-0.1) -- (1,6.5);
\draw[vertical] (2,-0.1) -- (2,6.5);
\draw[vertical] (3,-0.1) -- (3,6.5);
\draw[vertical] (4,-0.1) -- (4,6.5);
\draw[vertical] (5,-0.1) -- (5,6.5);
\draw[vertical] (6,-0.1) -- (6,6.5);
\draw[vertical] (7,-0.1) -- (7,6.5);
\draw[vertical] (8,-0.1) -- (8,6.5);
\draw[vertical] (9,-0.1) -- (9,6.5);
\draw[vertical] (10,-0.1) -- (10,6.5);

\draw[horizontal] (-0.1, 1.00000000000000) -- (10.5, 1.00000000000000);
\draw[horizontal] (-0.1, 2.00000000000000) -- (10.5, 2.00000000000000);
\draw[horizontal] (-0.1, 3.00000000000000) -- (10.5, 3.00000000000000);
\draw[horizontal] (-0.1, 4.00000000000000) -- (10.5, 4.00000000000000);
\draw[horizontal] (-0.1, 5.0) -- (10.5, 5.0);
\draw[horizontal] (-0.1, 6.0) -- (10.5, 6.0);

\draw (-0.2, 1.00000000000000) node[left] {$\unit{1}{ms}$};
\draw (-0.2, 2.00000000000000) node[left] {$\unit{2}{ms}$};
\draw (-0.2, 3.00000000000000) node[left] {$\unit{3}{ms}$};
\draw (-0.2, 4.00000000000000) node[left] {$\unit{4}{ms}$};
\draw (-0.2, 5.0) node[left] {$\unit{5}{ms}$};
\draw (-0.2, 6.0) node[left] {$\unit{6}{ms}$};

\draw (1, -0.2) node[below] {$1$};
\draw (2, -0.2) node[below] {$2$};
\draw (3, -0.2) node[below] {$3$};
\draw (4, -0.2) node[below] {$4$};
\draw (5, -0.2) node[below] {$5$};
\draw (6, -0.2) node[below] {$6$};
\draw (7, -0.2) node[below] {$7$};
\draw (8, -0.2) node[below] {$8$};
\draw (9, -0.2) node[below] {$9$};
\draw (10, -0.2) node[below] {$10$};

\draw (9, 0.973692694153) node[bull] {$\times$};
\draw (7, 0.395330355934) node[bull] {$\times$};
\draw (6, 0.446085782014) node[bull] {$\times$};
\draw (8, 0.588429881017) node[bull] {$\times$};
\draw (8, 0.460174805845) node[bull] {$\times$};
\draw (10, 0.811944170322) node[bull] {$\times$};
\draw (8, 0.817413609591) node[bull] {$\times$};
\draw (9, 0.526843139118) node[bull] {$\times$};
\draw (10, 1.31454460642) node[bull] {$\times$};
\draw (10, 1.44484377375) node[bull] {$\times$};
\draw (10, 0.743772299269) node[bull] {$\times$};
\draw (9, 1.00584272214) node[bull] {$\times$};
\draw (8, 0.497269578745) node[bull] {$\times$};
\draw (6, 0.2878647346) node[bull] {$\times$};
\draw (8, 2.02005026224) node[bull] {$\times$};
\draw (8, 0.6664039155) node[bull] {$\times$};
\draw (9, 0.645924330866) node[bull] {$\times$};
\draw (8, 0.368911003024) node[bull] {$\times$};
\draw (8, 0.707492075727) node[bull] {$\times$};
\draw (10, 0.665829046077) node[bull] {$\times$};
\draw (8, 0.502768108658) node[bull] {$\times$};
\draw (6, 0.411419868469) node[bull] {$\times$};
\draw (5, 0.439592770168) node[bull] {$\times$};
\draw (10, 0.680270967681) node[bull] {$\times$};
\draw (9, 0.730795901021) node[bull] {$\times$};
\draw (7, 1.21122362453) node[bull] {$\times$};
\draw (6, 0.426391298457) node[bull] {$\times$};
\draw (8, 0.422124067417) node[bull] {$\times$};
\draw (10, 0.797113034615) node[bull] {$\times$};
\draw (8, 0.784677451528) node[bull] {$\times$};
\draw (9, 0.558899124929) node[bull] {$\times$};
\draw (8, 0.557336432827) node[bull] {$\times$};
\draw (10, 1.65695749732) node[bull] {$\times$};
\draw (9, 0.878587659122) node[bull] {$\times$};
\draw (8, 0.465836098075) node[bull] {$\times$};
\draw (6, 0.781499978268) node[bull] {$\times$};
\draw (1, 0.252962112427) node[bull] {$\times$};
\draw (8, 0.506190461966) node[bull] {$\times$};
\draw (10, 1.14137567884) node[bull] {$\times$};
\draw (8, 0.915317507515) node[bull] {$\times$};
\draw (3, 0.355625152588) node[bull] {$\times$};
\draw (9, 0.665070923422) node[bull] {$\times$};
\draw (8, 0.377104217242) node[bull] {$\times$};
\draw (4, 0.682932989938) node[bull] {$\times$};
\draw (10, 1.39322487012) node[bull] {$\times$};
\draw (10, 0.776899942394) node[bull] {$\times$};
\draw (9, 0.494371139197) node[bull] {$\times$};
\draw (8, 0.499991833767) node[bull] {$\times$};
\draw (10, 5.94537033584) node[bull] {$\times$};
\draw (10, 0.916009973249) node[bull] {$\times$};
\draw (9, 0.611525898991) node[bull] {$\times$};
\draw (7, 0.715921574417) node[bull] {$\times$};
\draw (10, 1.52558454953) node[bull] {$\times$};
\draw (8, 0.423881386191) node[bull] {$\times$};
\draw (10, 1.67110318043) node[bull] {$\times$};
\draw (10, 0.759863611451) node[bull] {$\times$};
\draw (4, 0.456936219159) node[bull] {$\times$};
\draw (10, 1.19515079497) node[bull] {$\times$};
\draw (9, 0.507364993427) node[bull] {$\times$};
\draw (8, 0.549336930656) node[bull] {$\times$};
\draw (6, 0.433429409958) node[bull] {$\times$};
\draw (5, 0.495788898874) node[bull] {$\times$};
\draw (9, 1.38010559365) node[bull] {$\times$};
\draw (10, 0.632922078732) node[bull] {$\times$};
\draw (9, 1.0205599227) node[bull] {$\times$};
\draw (6, 0.443601980805) node[bull] {$\times$};
\draw (8, 0.457965458518) node[bull] {$\times$};
\draw (8, 0.701011969004) node[bull] {$\times$};
\draw (8, 0.971253567733) node[bull] {$\times$};
\draw (3, 0.321197509766) node[bull] {$\times$};
\draw (9, 0.545319915729) node[bull] {$\times$};
\draw (8, 0.476545939244) node[bull] {$\times$};
\draw (10, 0.672186697272) node[bull] {$\times$};
\draw (9, 1.04712628911) node[bull] {$\times$};
\draw (8, 0.517508440618) node[bull] {$\times$};
\draw (6, 0.441821970681) node[bull] {$\times$};
\draw (10, 0.736699374446) node[bull] {$\times$};
\draw (8, 2.15522942843) node[bull] {$\times$};
\draw (10, 0.91556085748) node[bull] {$\times$};
\draw (9, 0.753115084185) node[bull] {$\times$};
\draw (6, 0.770877708088) node[bull] {$\times$};
\draw (9, 0.570098637921) node[bull] {$\times$};
\draw (8, 0.368814800171) node[bull] {$\times$};
\draw (10, 0.981444500615) node[bull] {$\times$};
\draw (8, 0.562978601334) node[bull] {$\times$};
\draw (10, 0.723384490239) node[bull] {$\times$};
\draw (9, 0.446571239316) node[bull] {$\times$};
\draw (8, 0.575776836311) node[bull] {$\times$};
\draw (6, 0.357063345634) node[bull] {$\times$};
\draw (9, 2.85788609235) node[bull] {$\times$};
\draw (10, 1.06568436028) node[bull] {$\times$};
\draw (10, 0.651373859961) node[bull] {$\times$};
\draw (9, 0.711981695052) node[bull] {$\times$};
\draw (7, 1.29794185256) node[bull] {$\times$};
\draw (2, 0.286459922791) node[bull] {$\times$};
\draw (8, 0.435628258565) node[bull] {$\times$};
\draw (10, 0.560431674787) node[bull] {$\times$};
\draw (8, 0.746021655255) node[bull] {$\times$};
\draw (9, 0.471773093596) node[bull] {$\times$};
\draw (8, 0.55239145719) node[bull] {$\times$};
\draw (5, 0.354597965876) node[bull] {$\times$};
\draw (10, 0.564990144157) node[bull] {$\times$};
\draw (9, 0.875984676179) node[bull] {$\times$};
\draw (7, 0.350344373171) node[bull] {$\times$};
\draw (6, 0.697425071229) node[bull] {$\times$};
\draw (8, 0.445409857454) node[bull] {$\times$};
\draw (10, 1.14749215307) node[bull] {$\times$};
\draw (8, 0.916566818568) node[bull] {$\times$};
\draw (9, 0.656615752195) node[bull] {$\times$};
\draw (8, 0.390903919898) node[bull] {$\times$};
\draw (10, 4.24480353063) node[bull] {$\times$};
\draw (4, 0.625491142273) node[bull] {$\times$};
\draw (10, 1.0144863906) node[bull] {$\times$};
\draw (10, 0.734179246463) node[bull] {$\times$};
\draw (9, 0.865427459256) node[bull] {$\times$};
\draw (8, 0.46717645554) node[bull] {$\times$};
\draw (6, 0.327520012536) node[bull] {$\times$};
\draw (8, 0.593241929054) node[bull] {$\times$};
\draw (10, 0.894861471216) node[bull] {$\times$};
\draw (9, 0.643476114525) node[bull] {$\times$};
\draw (8, 0.410661613078) node[bull] {$\times$};
\draw (10, 0.834801359744) node[bull] {$\times$};
\draw (4, 0.459769192864) node[bull] {$\times$};
\draw (10, 0.732008411038) node[bull] {$\times$};
\draw (9, 0.50152801728) node[bull] {$\times$};
\draw (8, 0.571024866835) node[bull] {$\times$};
\draw (6, 0.351200712488) node[bull] {$\times$};
\draw (5, 0.911190396263) node[bull] {$\times$};
\draw (10, 1.16784237092) node[bull] {$\times$};
\draw (8, 0.687918933967) node[bull] {$\times$};
\draw (9, 0.765658805201) node[bull] {$\times$};
\draw (7, 0.310102841149) node[bull] {$\times$};
\draw (6, 0.425691604614) node[bull] {$\times$};
\draw (8, 0.422862143165) node[bull] {$\times$};
\draw (10, 0.818710080225) node[bull] {$\times$};
\draw (8, 0.783524076107) node[bull] {$\times$};
\draw (9, 0.609635719533) node[bull] {$\times$};
\draw (8, 0.557187367106) node[bull] {$\times$};
\draw (10, 1.50480692629) node[bull] {$\times$};
\draw (10, 0.526476141947) node[bull] {$\times$};
\draw (9, 1.02406199074) node[bull] {$\times$};
\draw (8, 0.505700778996) node[bull] {$\times$};
\draw (6, 0.531112106101) node[bull] {$\times$};
\draw (8, 1.33019611573) node[bull] {$\times$};
\draw (10, 0.892244448212) node[bull] {$\times$};
\draw (8, 0.89056679372) node[bull] {$\times$};
\draw (9, 3.00569207714) node[bull] {$\times$};
\draw (9, 0.571640371248) node[bull] {$\times$};
\draw (8, 0.374091111935) node[bull] {$\times$};
\draw (8, 0.587578689596) node[bull] {$\times$};
\draw (10, 0.744573536438) node[bull] {$\times$};
\draw (9, 0.443980614733) node[bull] {$\times$};
\draw (8, 0.502115688683) node[bull] {$\times$};
\draw (6, 0.294301074697) node[bull] {$\times$};
\draw (10, 0.858959022919) node[bull] {$\times$};
\draw (10, 0.904906356739) node[bull] {$\times$};
\draw (9, 0.805216214884) node[bull] {$\times$};
\draw (7, 0.496001758669) node[bull] {$\times$};
\draw (8, 0.420753137232) node[bull] {$\times$};
\draw (10, 1.66672092267) node[bull] {$\times$};
\draw (10, 0.830602034773) node[bull] {$\times$};
\draw (4, 0.502754660214) node[bull] {$\times$};
\draw (9, 0.473738798756) node[bull] {$\times$};
\draw (8, 0.47186983669) node[bull] {$\times$};
\draw (5, 0.452318053315) node[bull] {$\times$};
\draw (10, 0.521986917031) node[bull] {$\times$};

  \end{tikzpicture} 
  \caption{This figure plots, for $n\leq 10$ and for each transitive
    permutation group $G\subset \sg[n]$, the average computation time
    per canonical integer vectors below the staircase. If one ignores
    the two top groups for each $n$ (which correspond to $\mathfrak A_n$ and
    $\sg[n]$ respectively), the worst case complexity is seemingly
    roughly linear in the size of the result: $\OO(nC(G))$. }
\end{figure}
\TODO{improve by using $n!/|G|+Catalan(n)$ as reference?}
\begin{figure}
  \label{benchmarks.cannonics.bound}
  \centering
  \input{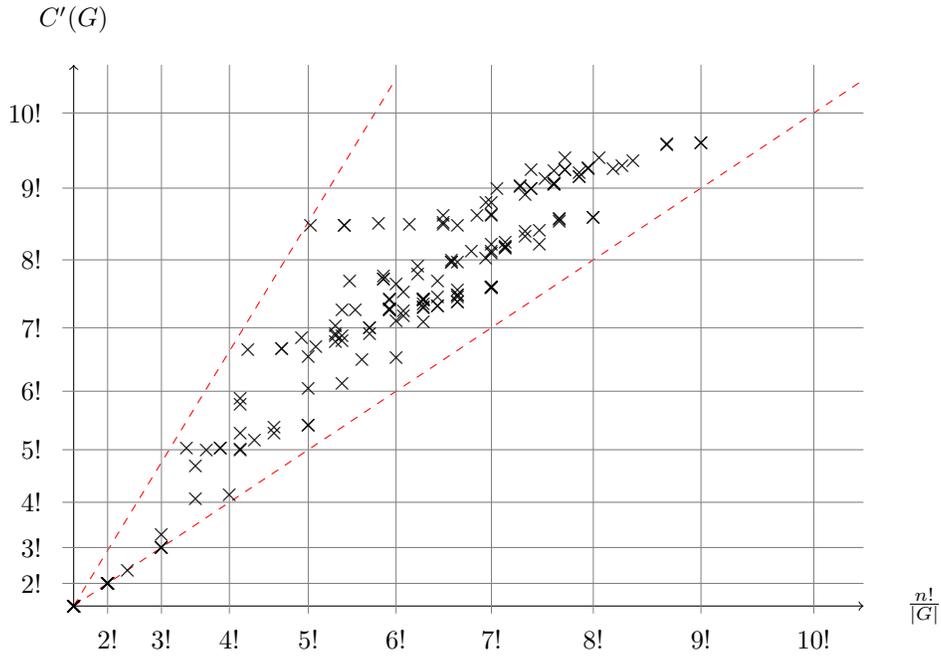}
  \caption{This figure plots, for $n\leq 10$ and for each transitive
    permutation group $G\subset \sg[n]$, the number
    $C'(G):=C(G)-\operatorname{catalan}(n)+1$ of canonical integer
    vectors below the staircase for $G$ which are \emph{not} non zero
    partitions, versus the number $n!/|G|$ of secondary
    invariants. The dotted lines suggest that, in practice, $n!/|G|
    \leq C'(G)\leq (n!/|G|)^{2.5}$.}
\end{figure}

% We conclude with a little optimization for the evaluation of the
% orbitsum of a monomial.
\begin{remark}
  \label{remark.evaluation.monomial}
  Let $\x^\alpha$ be a monomial. Then, evaluating it on a point
  $\rhoroot_\sigma$ requires at most $\OO(n)$ arithmetic operations in
  $\ZZ$. Assume indeed that $\rho^k$ has been precomputed in $\kk$ and
  cached for all $k$ in $0,\dots,n-1$; then, one can use:
  \begin{displaymath}
    \x^\alpha(\rhoroot_\sigma) = \rho^{\langle \alpha\mid \sigma\rangle \mod n }\,,
  \end{displaymath}
  where $\sigma$ is written, in the scalar product, as a permutation
  of $\{0,\dots,n-1\}$.
\end{remark}

\begin{remark}
  \label{remark.evaluation.orbitsum}
  Currently, the evaluation of the orbitsum $\orbitsum{\x^\alpha}$ of
  a monomial on a point $\rhoroot_\sigma$ is carried out by evaluating
  each monomial in the orbit. This gives a complexity of $\OO(n|G|)$
  arithmetic operations in $\ZZ$ (for counting how many times each
  $\rho^k$ appears in the result) and $\OO(n)$ additions in $\kk$ (for
  expressing the result in $\kk$). This can be roughly bounded by
  $\OO(|G|)$ arithmetic operations in $\kk$. This bounds the
  complexity of calculating $\Phi(\orbitsum{\x^\alpha})$ on all
  $\frac{n!}{|G|}$ points by $\frac{n!}{|G|} \OO(|G|)=\OO(n!)$.
\end{remark}

This worst case complexity gives only a very rough overestimate of the
average complexity in our application. Indeed, in practice, most of
the irreducible secondary invariants are of low degree; thus
Algorithm~\ref{algorithm} only need to evaluate orbitsums of monomials
$m$ of low degree; such monomials have many multiplicities in their
exponent vector, and tend to have a large automorphism group, that is
a small orbit.

Furthermore, it is to be expected that such evaluations can be carried
out much more efficiently by exploiting the inherent redundancy (\emph{à la}
Fast Fourrier Transform). In particular, one can use the strong
generating set of $G$ to apply a divide and conquer approach to the
evaluation of an orbitsum on a point. The complexity analysis and
benchmarking remains to be done to evaluate the practical
gain. Finally, the evaluation of an orbitsum on many points is
embarrassingly parallel (though fine grained), a property which we
have not exploited yet.

\section{Complexity analysis}
\label{section.complexity}

For the sake of simplicity, all complexity results are expressed in
terms of arithmetic operations in the ground field
$\kk=\QQ(\rho)$. This model is realistic, because, in practice, the
growth of coefficients does not seem to become a bottleneck; a
possible explanation for this phenomenon might be that the natural
coefficient growth would be compensated by the pointwise product which
tends to preserve and increase sparseness.  We also consider that one
operation in $\kk$ is equivalent to $n$ operations in $\QQ$. This is a
slight abuse; however $\dim_\QQ \QQ(\rho) = \phi(n)\geq 0.2n$ for
$n\leq 10000$ which is far beyond any practical value of $n$ in our
context.

\begin{theorem}
  \label{proposition.complexity}
  Let $G$ be a permutation group, and take the elementary symmetric
  functions as primary invariants. Then, the complexity of computing
  secondary invariants by evaluation using Algorithm~\ref{algorithm}
  is bounded above by $\OO(n!^2 + n!^3/|G|^2)$ arithmetic operations
  in $\kk$.
\end{theorem}
\begin{proof}
  To get this upper bound on the complexity, we broadly simplify the
  main steps of this algorithm to:
  \begin{enumerate}
  \item Group theoretic computations on $G$: strong generating set,
    conjugacy classes, etc;
    \label{proposition.complexity.group}
  \item Computation of the Hilbert series of $\RG$;
    \label{proposition.complexity.hilbert}
  \item Construction of canonical monomials under the staircase;
    \label{proposition.complexity.canonical}
  \item Computation by $\Phi$ of the evaluation vectors of the
    orbitsums of those monomials;
    \label{proposition.complexity.evaluation}
  \item Computation of products $\Phi(\eta)\Phi(\eta')$ of evaluation
    vectors of secondary invariants;
    \label{proposition.complexity.products}
  \item Row reduction of the evaluation vectors.
    \label{proposition.complexity.reduction}
  \end{enumerate}

  The complexity of \eqref{proposition.complexity.group} is a small
  polynomial in $n$ (see
  e.g.~\cite{Seress.2003.PermiutationGroupAlgorithms}) and is
  negligible in practice as well as in
  theory. \eqref{proposition.complexity.hilbert} can be reduced to the
  addition of $c$ polynomials of degree at most $\binom n 2$, where
  $c\leq |G|\leq n!$ is the number of conjugacy classes of $G$ (the
  denominator of the Hilbert series is known; the mentioned
  polynomials contribute to its numerator, that is the generating
  series of the secondary invariants); it is negligible as
  well. Furthermore, by Remark~\ref{remark.canonical}
  \eqref{proposition.complexity.canonical} is not a bottleneck.

  Using Lemma~\ref{lemma.spanning} and
  Remark~\ref{remark.evaluation.orbitsum}, the complexity of
  \eqref{proposition.complexity.evaluation} is bounded above by
  $\OO(n!^2)$ (at most $\OO(n!)$ orbitsums to evaluate, for a cost of
  $\OO(n!)$ each).

  For a very crude upper bound for
  \eqref{proposition.complexity.products}, we assume that the
  algorithm computes all products of evaluation vectors of two
  secondary invariants. This gives $(n!/G)^2$ products in $\EG$ which
  is in $\OO(n!/G)^3$.

  Finally, in \eqref{proposition.complexity.reduction}, the cost of
  the row reduction of $\OO(n!)$ evaluation vectors in $\EG$ is of
  $\OO(n!^3/|G|^2)$.
  %
  % We have $n!/|G|$ different evaluation points on which we can
  % evaluate at most $n!$ orbit sums (This correspond to take all
  % monomials under the staircase).  Each orbit sum present at most
  % $|G|$ elements which we will have to sum. The evaluation present a
  % cost of $n!|G| \cdot n!/|G| = (n!)^2$ arithmetic operations. We can
  % consider that evaluate orbit sums doesn't produce multiplications :
  % from an integer vector representing a monomial, the multiplication
  % consist in computing the sum of the entries, reduced it modulo $n$
  % and take the corresponding power of $\rho$.
  %
  % Once we have evaluated all orbit sums, we have to make a Gauss
  % reduction over it. In practice, this step is optimized by doing it
  % degree by degree.  If we do this operation for all evaluations in
  % the same time, we reduced a rectangular matrix with $n!$ columns
  % (candidates to be secondary invariants) and $n!/|G|$ lines (number
  % of evaluation points). We know mathematically that we can extract a
  % family of $n!/|G|$ independents vectors which will correspond to a
  % system of secondary invariants. it follows from that a cost of
  % $n!*(n!/|G|)^2 = n!^3/|G|^2$ arithmetic operations over $\Cn$.
\end{proof}

This complexity bound gives some indication that the symmetries are
honestly taken care of by this algorithm. Consider indeed any
algorithm computing secondary invariants by linear algebra in
$\KK[x]/\Sym^+$ (say using Gröbner basis or orthogonal bases for the
Schur-Schubert scalar product). Then the same estimation gives a
complexity of $\OO(n!^2/|G|)$ (reducing $n!$ candidates to get
$n!/|G|$ linearly independent vectors in a vector space of dimension
$n!$).  Therefore, for $G$ large enough, a gain of $|G|$ is obtained.

That being said, this is a \emph{very} crude upper bound. For a fixed
group $G$, one could use the Hilbert series to calculate explicitly a
much better estimate: indeed the grading splits the linear algebra in
many smaller problems and also greatly reduces the number of products
to consider. However, it seems hard in general to get enough control
on the Hilbert series, to derive complexity information solely in term
of basic information on the group ($n$, $|G|$, ...).
Also, in practice, there usually are only few irreducible invariants,
and they are of small degrees. Thus only few of the canonical monomial
need actually to be generated and evaluated.

It is therefore essential to complement this complexity analysis with
extensive benchmarks to confirm the practical gains. This is the topic
of the next section.

\section{Implementation and benchmarks}
\label{section.benchmarks}

Algorithm~\ref{algorithm}, and many variants, have been implemented in
the open source mathematical platform \Sage~\cite{Sage}. The choice of
the platform was motivated by the availability of most of the basic
tools (group theory via \gap~\cite{GAP4}, cyclotomic fields, linear
algebra, symmetric functions, etc), and the existence of a community
to share with the open-source development of the remaining tools
(e.g. Schubert polynomials or the orderly generation of canonical
monomials)~\cite{Sage-Combinat}. Thanks to the \cython compiler, it
was also easy to write most of the code in a high level interpreted
language (\python), and cherry pick just those critical sections that
needed to be compiled (orderly generation, evaluation). The
implementation is publicly available in alpha version via the
\sagecombinat patch server. It will eventually be integrated into the
\Sage library. \TODO{\#lines of code}

We ran systematic benchmarks (see Figure~\ref{benchmark.sage-singular}
and~\ref{benchmark.relative}), comparing the results with the
implementation of secondary invariants in
\singular~\cite{Singular,King.2007.secondary}. Note that \singular's
implementation deals with any finite group of matrices. Also, it
precomputes and uses its own primary invariants instead of the
elementary symmetric functions. Therefore, the comparison is not
immediate: on the one hand, \singular has more work to do (finding the
primary invariants); on the other hand, when the primary invariants
are of small degree, the size of the result can be much smaller. Thus,
those benchmarks should eventually be complemented by:
\begin{itemize}
\item Calculations of secondary invariants w.r.t. the elementary
  symmetric functions, using Gröbner basis using \singular and
  \magma;\TODO{ or \Sage}
\item Calculations of secondary invariants using \singular and
  \magma;\TODO{ or \Sage}
\item Calculations of secondary invariants w.r.t. the elementary
  symmetric functions, using SAGBI-Gröbner basis (for example by using
  \mupadcombinat~\cite{Thiery.PerMuVAR,MuPAD-Combinat}).
\end{itemize}
A similar benchmark comparing \magma~\cite{Cannon_al.1996} and \mupadcombinat
is presented in~\cite[Figure~1]{Thiery.CMGS.2001} (up to a bias: the
focus in \mupadcombinat is on a minimal generating set, but this is
somewhat equivalent to irreducible secondary invariants). This
benchmark can be roughly compared with that of
Figure~\ref{benchmark.sage-singular} by shifting by a speed factor of
$10$ to compensate for the hardware improvements since 2001. Related
benchmarks are available
in~\cite{King.2007.secondary,King.2007.minimal}.

We used the transitive permutation groups as test bed. A practical
motivation is that there are not so many of them and they are easily
available through the \gap
database~\cite{Hulpke.2005.TransitivePermutationGroups}. At the same
time, we claim that they provide a wide enough variety of permutation
groups to be representative. In particular, the computation for non
transitive permutation groups tend to be easier, since one can use
primary invariants of much smaller degrees, namely the elementary
symmetric functions in each orbit of variables.

The benchmarks were run on the computation server
\texttt{sage.math.washington.edu}\footnote{This server is part of the
  \Sage cluster at the University of Washington at Seattle and is
  devoted to \Sage development; it was financed by "National Science
  Foundation Grant No. DMS-0821725".} which is equipped with 24
Intel(R) Xeon(R) CPU X7460 @2.66GHz cores and $128$~GB of RAM. We did
not use parallelism, except for running up to four tests in parallel.
The memory usage is fairly predictable, at least for the \Sage
implementation, so we did not include it into the benchmarks. In
practice, the worst calculation used $12$~GB. Any calculation running
over 24 hours was aborted.

%\begin{landscape}
\begin{figure}
  \label{benchmark.sage-singular}
  \begin{bigcenter}
    \scalebox{1}{\input{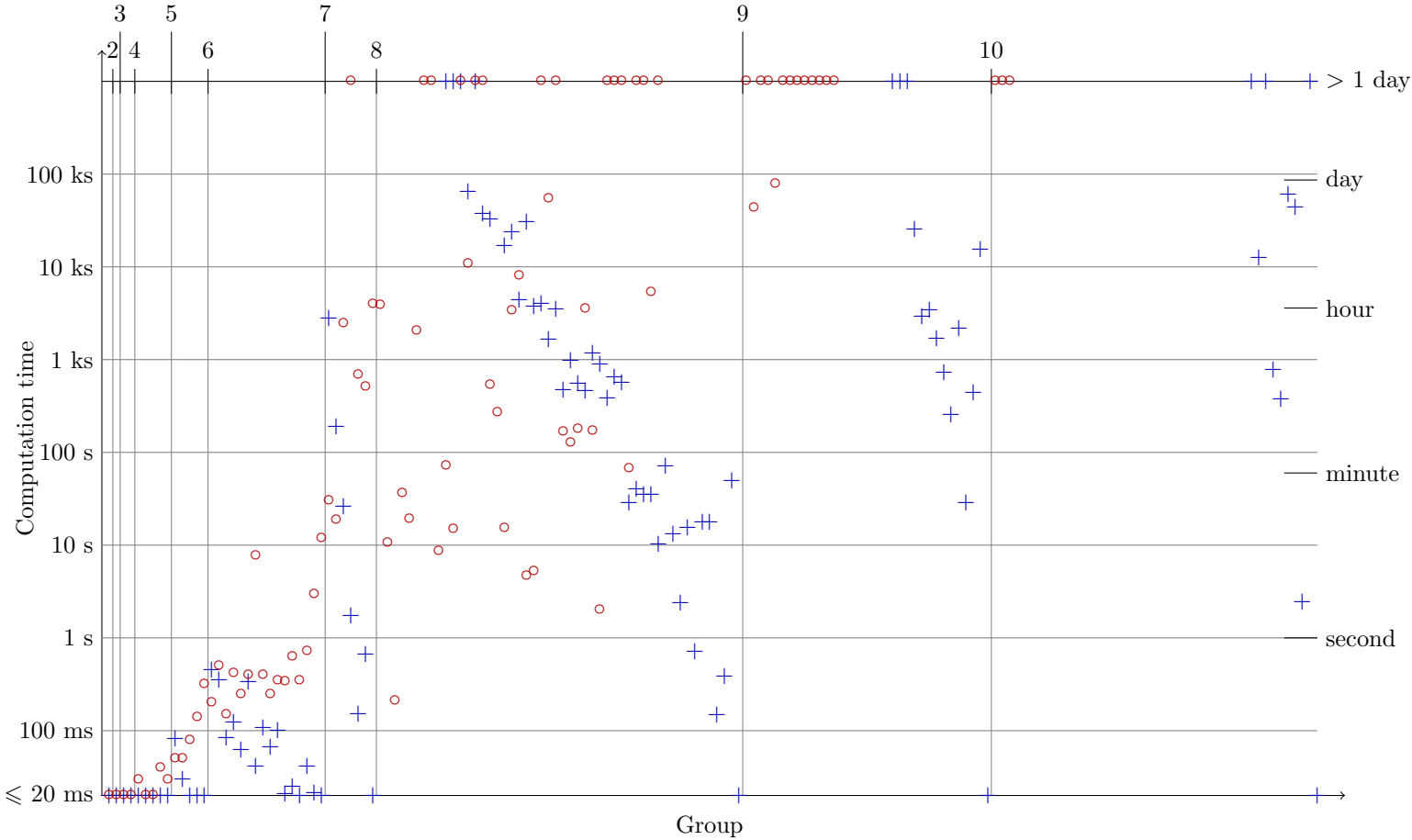}}
  \end{bigcenter}
  \caption[Comparative benchmark between \Sage and
  \singular]{Comparative benchmark for the computation of secondary
    invariants for all transitive permutation groups for $n\leq 10$
    using \Sage's evaluation implementation ($\color{darkblue}{+}$)
    and \singular's elimination implementation
    ($\color{darkred}{\circ}$). The groups are sorted horizontally by
    increasing $n$ and then by increasing cardinality.}
\end{figure}
%\end{landscape}

\TODO{Add relevant data or each group? n, group id, |G|, number of
  secondary invariants, time.}

%\begin{landscape}
\begin{figure}
  \label{benchmark.relative}
  \begin{bigcenter}
    \scalebox{0.8}{  \begin{tikzpicture}[xscale=0.8, yscale=0.9]
    \tikzstyle{sage}=[darkblue, minimum size=4pt, inner sep=0pt] 
    \tikzstyle{horizontal}=[-, color=gray] 
    \tikzstyle{vertical}=[very thin, color=gray] 
      % axis 
      \draw[->] (0,0) -- (22.5,0); 
      \draw[->] (0,0) -- (0,13.5); 

      \draw (22.6, 0) node[right] {$\frac{n!}{|G|}$}; 
      \draw[horizontal] (0,1.29885434651241) -- (22,1.29885434651241); 
      \draw[horizontal] (0,2.59770869302482) -- (22,2.59770869302482); 
      \draw[horizontal] (0,3.89656303953723) -- (22,3.89656303953723); 
      \draw[horizontal] (0,5.19541738604964) -- (22,5.19541738604964); 
      \draw[horizontal] (0,6.49432474702518) -- (22,6.49432474702518); 
      \draw[horizontal] (0,7.79312607907446) -- (22,7.79312607907446); 
      \draw[horizontal] (0,9.09208645451312) -- (22,9.09208645451312); 
      \draw[horizontal] (0,10.3908347720993) -- (22,10.3908347720993); 
      \draw[horizontal] (0,11.6896891186117) -- (22,11.6896891186117); 

      \draw (-1.5, 5.0) node[rotate=90, above] {Computation time}; 

      \draw (0,0) node[left] {$\leqslant  100 \mu s$}; 
      \draw (0,1.29885434651241) node[left] {1 ms}; 
      \draw (0,2.59770869302482) node[left] {10 ms}; 
      \draw (0,3.89656303953723) node[left] {100 ms}; 
      \draw (0,5.19541738604964) node[left] {1 s}; 
      \draw (0,6.49432474702518) node[left] {10 s}; 
      \draw (0,7.79312607907446) node[left] {100 s}; 
      \draw (0,9.09208645451312) node[left] {1 ks}; 
      \draw (0,10.3908347720993) node[left] {10 ks}; 
      \draw (0,11.6896891186117) node[left] {100 ks}; 

      \draw (0,12.9885692415184) -- (22,12.9885692415184) node[right] {> 1 day}; 

      \draw (21.4,11.6072509246740) -- (22,11.6072509246740) node[right] {day}; 
      \draw (21.4,9.81455162479343) -- (22,9.81455162479343) node[right] {hour}; 
      \draw (21.4,7.50498450542154) -- (22,7.50498450542154) node[right] {minute}; 
      \draw (21.4,5.19541738604964) -- (22,5.19541738604964) node[right] {second}; 
      \draw (21.4,1.29885434651241) -- (22,1.29885434651241) node[right] {millisecond}; 
      \draw (0,0) node[below] {$1!$}; 
      \draw (1.19116672661692,0) node[below] {$2!$}; 
      \draw[vertical] (1.19116672661692,0) -- (1.19116672661692,12.9885692415184); 
      \draw (3.07912132041150,0) node[below] {$3!$}; 
      \draw[vertical] (3.07912132041150,0) -- (3.07912132041150,12.9885692415184); 
      \draw (5.46145477364533,0) node[below] {$4!$}; 
      \draw[vertical] (5.46145477364533,0) -- (5.46145477364533,12.9885692415184); 
      \draw (8.22725826187216,0) node[below] {$5!$}; 
      \draw[vertical] (8.22725826187216,0) -- (8.22725826187216,12.9885692415184); 
      \draw (11.3063795822837,0) node[below] {$6!$}; 
      \draw[vertical] (11.3063795822837,0) -- (11.3063795822837,12.9885692415184); 
      \draw (14.6504073552429,0) node[below] {$7!$}; 
      \draw[vertical] (14.6504073552429,0) -- (14.6504073552429,12.9885692415184); 
      \draw (18.2239075350937,0) node[below] {$8!$}; 
      \draw[vertical] (18.2239075350937,0) -- (18.2239075350937,12.9885692415184); 
      \draw (21.9998167226828,0) node[below] {$9!$}; 
      \draw[vertical] (21.9998167226828,0) -- (21.9998167226828,12.9885692415184);

      \draw (3.95697021484375, 0.2) -- (3.95697021484375, -0.4) node[below] {$10^1$}; 
      \draw (7.91394042968750, 0.2) -- (7.91394042968750, -0.4) node[below] {$10^2$}; 
      \draw (11.8709106445312, 0.2) -- (11.8709106445312, -0.4) node[below] {$10^3$}; 
      \draw (15.8278808593750, 0.2) -- (15.8278808593750, -0.4) node[below] {$10^4$}; 
      \draw (19.7848510742187, 0.2) -- (19.7848510742187, -0.4) node[below] {$10^5$};

      \draw (0.0000, 0.9407)  node[sage] {1}; 
      \draw (0.0000, 0.9331)  node[sage] {2}; 
      \draw (1.191, 2.291)  node[sage] {3}; 
      \draw (0.0000, 0.9658)  node[sage] {3}; 
      \draw (3.079, 2.617)  node[sage] {4}; 
      \draw (3.079, 2.349)  node[sage] {4}; 
      \draw (1.888, 1.987)  node[sage] {4}; 
      \draw (1.191, 2.055)  node[sage] {4}; 
      \draw (0.0000, 0.9078)  node[sage] {4}; 
      \draw (5.461, 3.789)  node[sage] {5}; 
      \draw (4.270, 3.216)  node[sage] {5}; 
      \draw (3.079, 2.871)  node[sage] {5}; 
      \draw (1.191, 2.426)  node[sage] {5}; 
      \draw (0.0000, 0.9535)  node[sage] {5}; 
      \draw (8.227, 4.750)  node[sage] {6}; 
      \draw (8.227, 4.605)  node[sage] {6}; 
      \draw (7.036, 3.804)  node[sage] {6}; 
      \draw (7.036, 4.016)  node[sage] {6}; 
      \draw (6.339, 3.638)  node[sage] {6}; 
      \draw (5.845, 4.587)  node[sage] {6}; 
      \draw (5.845, 3.396)  node[sage] {6}; 
      \draw (5.845, 3.943)  node[sage] {6}; 
      \draw (5.148, 3.671)  node[sage] {6}; 
      \draw (5.148, 3.905)  node[sage] {6}; 
      \draw (4.654, 3.016)  node[sage] {6}; 
      \draw (4.270, 3.119)  node[sage] {6}; 
      \draw (3.957, 2.753)  node[sage] {6}; 
      \draw (3.079, 3.397)  node[sage] {6}; 
      \draw (1.191, 3.029)  node[sage] {6}; 
      \draw (0.0000, 1.084)  node[sage] {6}; 
      \draw (11.31, 9.679)  node[sage] {7}; 
      \draw (10.12, 8.151)  node[sage] {7}; 
      \draw (9.418, 7.041)  node[sage] {7}; 
      \draw (8.227, 5.514)  node[sage] {7}; 
      \draw (5.845, 4.130)  node[sage] {7}; 
      \draw (1.191, 4.966)  node[sage] {7}; 
      \draw (0.0000, 1.120)  node[sage] {7}; 
      \draw (13.46, 12.99)  node[sage] {8}; 
      \draw (13.46, 12.99)  node[sage] {8}; 
      \draw (12.76, 12.99)  node[sage] {8}; 
      \draw (12.76, 11.45)  node[sage] {8}; 
      \draw (12.76, 12.99)  node[sage] {8}; 
      \draw (12.27, 11.14)  node[sage] {8}; 
      \draw (12.27, 11.06)  node[sage] {8}; 
      \draw (12.27, 10.68)  node[sage] {8}; 
      \draw (12.27, 10.88)  node[sage] {8}; 
      \draw (12.27, 9.924)  node[sage] {8}; 
      \draw (12.27, 11.02)  node[sage] {8}; 
      \draw (12.27, 9.839)  node[sage] {8}; 
      \draw (11.57, 9.875)  node[sage] {8}; 
      \draw (11.57, 9.384)  node[sage] {8}; 
      \draw (11.31, 9.796)  node[sage] {8}; 
      \draw (11.08, 8.676)  node[sage] {8}; 
      \draw (11.08, 9.088)  node[sage] {8}; 
      \draw (11.08, 8.761)  node[sage] {8}; 
      \draw (11.08, 8.664)  node[sage] {8}; 
      \draw (11.08, 9.190)  node[sage] {8}; 
      \draw (11.08, 9.034)  node[sage] {8}; 
      \draw (10.38, 8.560)  node[sage] {8}; 
      \draw (10.38, 8.855)  node[sage] {8}; 
      \draw (10.38, 8.779)  node[sage] {8}; 
      \draw (9.886, 7.084)  node[sage] {8}; 
      \draw (9.418, 7.287)  node[sage] {8}; 
      \draw (9.418, 7.202)  node[sage] {8}; 
      \draw (9.189, 7.201)  node[sage] {8}; 
      \draw (9.189, 6.506)  node[sage] {8}; 
      \draw (9.189, 7.599)  node[sage] {8}; 
      \draw (9.189, 6.654)  node[sage] {8}; 
      \draw (8.492, 5.694)  node[sage] {8}; 
      \draw (8.227, 6.746)  node[sage] {8}; 
      \draw (7.998, 5.001)  node[sage] {8}; 
      \draw (7.301, 6.815)  node[sage] {8}; 
      \draw (7.301, 6.820)  node[sage] {8}; 
      \draw (6.110, 4.122)  node[sage] {8}; 
      \draw (5.845, 4.654)  node[sage] {8}; 
      \draw (1.191, 7.401)  node[sage] {8}; 
      \draw (0.0000, 0.9179)  node[sage] {8}; 
      \draw (13.26, 12.99)  node[sage] {9}; 
      \draw (13.26, 12.99)  node[sage] {9}; 
      \draw (12.76, 12.99)  node[sage] {9}; 
      \draw (12.07, 10.92)  node[sage] {9}; 
      \draw (12.07, 9.701)  node[sage] {9}; 
      \draw (11.57, 9.783)  node[sage] {9}; 
      \draw (11.31, 9.391)  node[sage] {9}; 
      \draw (10.87, 8.916)  node[sage] {9}; 
      \draw (10.87, 8.323)  node[sage] {9}; 
      \draw (10.87, 9.531)  node[sage] {9}; 
      \draw (9.683, 7.093)  node[sage] {9}; 
      \draw (9.418, 8.630)  node[sage] {9}; 
      \draw (1.191, 10.64)  node[sage] {9}; 
      \draw (0.0000, 1.084)  node[sage] {9}; 
      \draw (12.96, 12.99)  node[sage] {10}; 
      \draw (12.96, 10.53)  node[sage] {10}; 
      \draw (12.96, 12.99)  node[sage] {10}; 
      \draw (11.77, 8.949)  node[sage] {10}; 
      \draw (10.69, 8.538)  node[sage] {10}; 
      \draw (9.502, 11.41)  node[sage] {10}; 
      \draw (9.502, 11.23)  node[sage] {10}; 
      \draw (8.311, 5.707)  node[sage] {10}; 
      \draw (1.191, 12.99)  node[sage] {10}; 
      \draw (0.0000, 1.199)  node[sage] {10}; 
      \draw (12.93, 12.99)  node[sage] {12}; 
      \draw (11.74, 12.99)  node[sage] {12}; 
      \draw (11.74, 12.99)  node[sage] {12}; 
      \draw (10.54, 7.598)  node[sage] {12}; 
      \draw (0.0000, 1.392)  node[sage] {12}; 
      \draw (13.99, 12.99)  node[sage] {14}; 
      \draw (12.80, 10.82)  node[sage] {14}; 
  \end{tikzpicture}}
  \end{bigcenter}
  \caption[Benchmark versus the number of secondary
  invariant]{Benchmark for the computation of secondary invariants for
    all transitive permutation groups for $n\leq 10$ (and for some
    below $n\leq 14$), using \Sage's evaluation implementation. For
    each such group, $n$ is written at position $(k,t)$, where
    $k=n!/|G|$ is the number of secondary invariants and $t$ is the
    computation time. In particular, the symmetric groups $\sg[n]$ and
    the alternating groups $\mathfrak A_n$ are respectively above $1!$
    and $2!$.}
\end{figure}
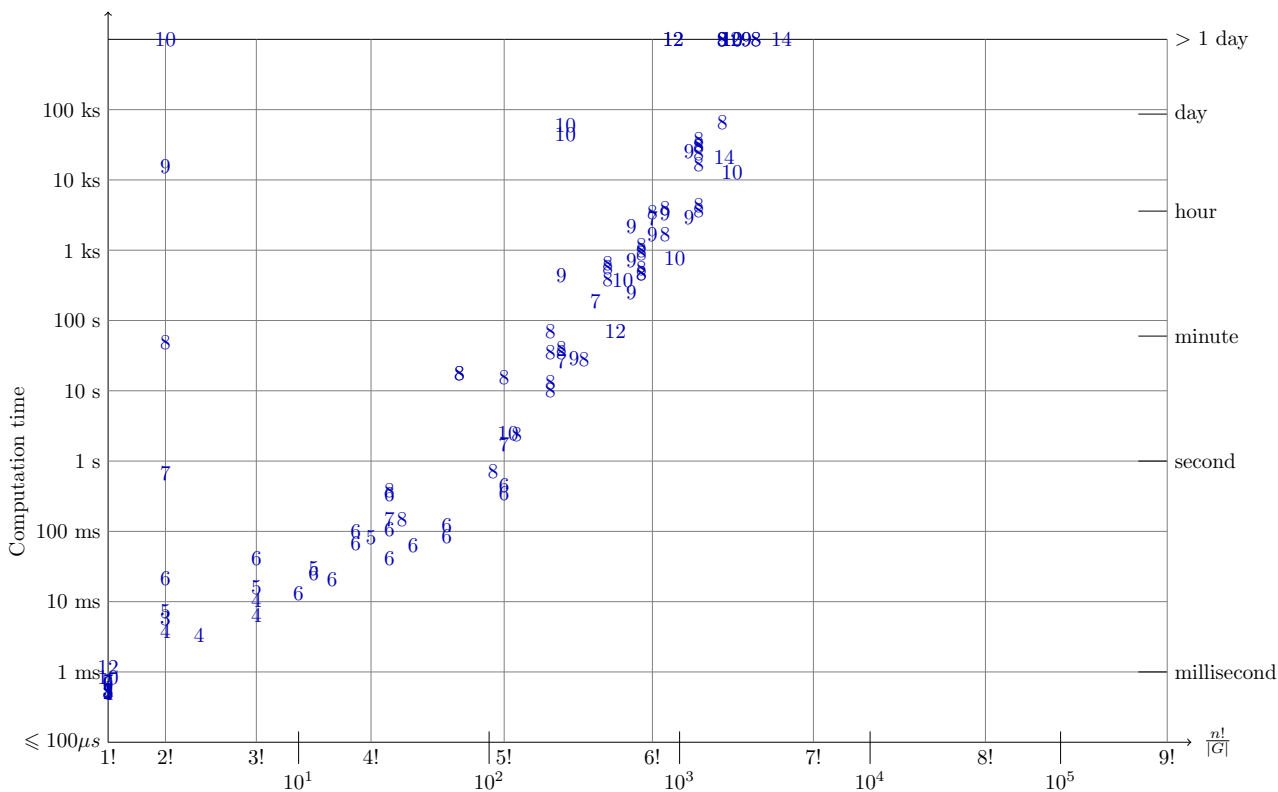
%\end{landscape}
%\end{sidewaysfigure}

\TODO{Add a line with the current theoretical complexity bound to
  highlight that it is vastly overestimated?}

\section{Further developments}
\label{section.future}

At this stage, the above sections validate the potential of the
evaluation approach. Yet much remains to be done, both in theory and
practice, to design algorithms making an optimal use of this approach.
The main bottleneck so far is the calculation of evaluations by
$\Phi$, and we conclude with a couple problems we are currently
investigating in this direction.

\begin{problem}
  Construct invariants with nice properties under evaluation by $\Phi$
  (sparsity, ...). A promising starting point are Schubert
  polynomials~\cite{LascouxSchutzenberger82,Lascoux.2003.CBMS}, as
  they form a basis of $\R$ as $\Sym$-module whose image under $\Phi$
  is triangular. However, it is not clear whether this triangularity
  can be made somehow compatible with the coset distribution of $G$ in
  $\sg[n]$.

  Another approach would be to search for invariants admitting short
  Straight Line Programs.
\end{problem}

Note that a good solution to this problem, combined with the
evaluation approach of this paper, could possibly open the door for
the solution of a long standing problem, namely the \emph{explicit}
construction of secondary invariants; currently such a description is
known only in the very simple case of products of symmetric
groups~\cite{Garsia_Stanton.1984}. Even just associating in some
canonical way a secondary invariant to each coset in $\sg[n]/G$ seems
elusive.

From a practical point of view, the following would be needed.
\begin{problem}
  Find a good algorithm to compute $\Phi$ on the above
  invariants. This is similar in spirit to finding an analogue of the
  Fast Fourier Transform w.r.t. the Fourier Transform.
\end{problem}

% $\Phi$ as a lot of combinatorial properties, this morphism has strong
% connections with the Schubert polynomials which constitute a nice basis
% of multivariate polynomials as free module over symmetric polynomials.
% As there are more points than needed in a lot of case, the following
% problem come out.

Theorem~\ref{theorem.secondary.evaluation} further suggests that,
using the grading, it could be sufficient to consider only a subset of
the evaluation
points.% to distinguish the secondary invariants. Indeed, for
%each degree $d$, one only need to consider simultaneously the
%secondary invariants of degree $d$, $d-n$, \dots.
This is corroborated by computer exploration; for example, for the
cyclic group $C_7$ of order $7$, $110$ evaluation points out of $720$
were enough for constructing the secondary invariants. Possible
approaches include lazy evaluation strategies, or explicit choices of
evaluation points, or some combination of both.
\begin{problem}
  \label{problem.points}
  Get some theoretical control on which evaluation points are needed
  so that $\Phi$ restricted on those points remains injective on some
  (resp. all) homogeneous component $\RG_d$.
\end{problem}
Here again, Schubert polynomials are natural candidates, with the same
difficulty as above.
A step toward Problem~\ref{problem.points}  would be to solve the following.
\begin{problem}
  For $G\subset\sg[n]$ a permutation group, and to start with for $G$
  the trivial permutation group, find a good description of the
  subspaces $\Phi(\RG_d)$.
\end{problem}

% The linear algebra appearing when using the morphism $\Phi$ present also very
% singular specifies. As we have nice Hilbert series, we know a priori the
% dimension we have to reach. When $n$ is big and the order of the permutation
% group small (cyclic groups are good examples), we search for a small number
% of independents vectors in a very big space. In top of that, there is a non
% negligeable cost for the evaluations. Our prototype currently use a dense
% data structure for vectors but it deserves to use a lazy Gauss reduction
% which ask for the evaluations only when it is needed. It will also
% have the benefit to determine the number of points needed to select the
% secondary invariants.

Last but not least, one would want to generalize the evaluation
approach to any matrix groups, following the line sketched in the
introduction. The issue is whether one can get enough control on
perturbations of the primary invariants so that:
\begin{itemize}
\item The orbits of the simple roots are large, in order to benefit
  from the gain of taking a single evaluation point per orbit;
\item Only few of the primary invariants need to be perturbated, to
  best exploit the grading in the analogue of
  Theorem~\ref{theorem.secondary.evaluation}.
\end{itemize}

\section{Acknowledgments}

We would like to thank Marc Giusti, Alain Lascoux, Romain Lebreton,
and Éric Schost, for fruitful discussions, as well as the anonymous
referees of the extended abstract presented at MEGA
2011~\cite{Borie_Thiery.2011.Invariants} for their many useful
suggestions for improvements.

This research was driven by computer exploration using the open-source
mathematical software \Sage~\cite{Sage}. In particular, we perused its
algebraic combinatorics features developed by the \sagecombinat
community~\cite{Sage-Combinat}, as well as its group theoretical and
invariant theoretical features provided respectively by
\gap~\cite{GAP} and \singular~\cite{Singular}. The extensive
benchmarks were run on the computational server
\texttt{sage.math.washington.edu}, courtesy of the \Sage developers
group at the University of Washington (Seattle, USA) and the "National
Science Foundation Grant No. DMS-0821725".

\bibliographystyle{alpha}

\bibliography{main}

\end{document}